\documentclass{amsart}

\usepackage{graphicx,amssymb,amsmath,xcolor,enumerate,latexsym}
\usepackage{hyperref}
\hypersetup{
bookmarks=true,         
pdffitwindow=false,     
pdfstartview={FitH},    
colorlinks=true,      
citecolor=red,
}

\usepackage{enumitem}

\usepackage{float}

\usepackage{mathrsfs} 
\usepackage[top=3cm, bottom=3cm, left=2.8cm, right=2.8cm]{geometry} 

\makeatletter
\def\l@subsection{\@tocline{2}{0pt}{2.5pc}{5pc}{}}
\makeatother

\DeclareSymbolFont{largesymbol}{OMX}{yhex}{m}{n}
\DeclareMathAccent{\Widehat}{\mathord}{largesymbol}{"62}
\newcommand*\di{\mathop{}\!\mathrm{d}}

\def\e{\epsilon}

\numberwithin{equation}{section}              
\newtheorem{theorem}{Theorem}[section]

\newtheorem{lemma}{Lemma}[section]
\newtheorem{proposition}{Proposition}[section]
\newtheorem*{proposition*}{Proposition}

\newtheorem*{corollary*}{Corollary}
\newtheorem{definition}{Definition}[section]
\newtheorem*{definitions*}{Definitions}

\newtheorem*{acknowledgements*}{Acknowledgements}

\newtheorem*{conjecture*}{\bf Conjecture}

\newtheorem*{example*}{\bf Example}
\theoremstyle{remark}
\newtheorem{remark}{\bf Remark}[section]

\begin{document}
\date{}                                     

\author{Cong Wang}
\address[C. Wang]{Department of Mathematics, Harbin Institute of Technology,  Harbin,	150001, P.R. China.}
\email{math\_congwang@163.com}

\author{Yu Gao}
\address[Y. Gao]{Department of Applied Mathematics, The Hong Kong Polytechnic University, Hung Hom, Kowloon, Hong Kong}
\email{mathyu.gao@polyu.edu.hk}

\author{Xiaoping Xue}
\address[X. Xue]{Department of Mathematics, Harbin Institute of Technology, Harbin,	150001, P.R. China.}
\email{xiaopingxue@hit.edu.cn}

\title[Space-time analyticity of NS]{Joint space-time analyticity of mild solutions to the Navier-Stokes equations}

\begin{abstract}
In this paper, we show the optimal decay rate estimates of the space-time derivatives and the joint space-time analyticity of solutions to the Navier-Stokes equations. As it is known from the Hartogs's theorem, for a complex function with two complex variables, the joint analyticity with respect to two variables can be derived from combining of analyticity with respect to each variable. However, as a function of two real variables for space and time, the joint space-time analyticity of solutions to the Navier-Stokes equations cannot be directly obtained from the combination of space analyticity and time analyticity. Our result seems to be the first quantitative result for the joint space-time analyticity of solutions to the Navier-Stokes equations, and the proof only involves real variable methods. Moreover, the decay rate estimates also yield the bounds on the growth (in time) of radius of space analyticity, time analyticity, and joint space-time analyticity of solutions. 
\end{abstract}

\maketitle

\section{Introduction}
The study of analyticity of solutions to partial differential equations has been a long history, and there are many applications of the analyticity of solutions, such as the solvability of backward equations \cite{zhang2020anote} and the control theory \cite{zhang2015observation,canzhang2017analyticity}. In the fluid dynamics, the radius of spatial analyticity can be used to measure the geometrically significant length scale of fluid flow \cite{grujic2001the} and to obtain Hausdorff length upper bounds of Navier-Stokes equations \cite{kukavica1996level}. Moreover, the analyticity of solutions accounts for the exponential convergence of the finite dimensional Galerkin method in the Ginzburg-Landau equation \cite{titi1993regularity}.
 
In this paper, we are going to study  the decay rate estimates for the space-time derivatives and the joint space-time analyticity of solutions to the incompressible Navier-Stokes equations in $\mathbb{R}^3$:
\begin{equation}\label{eq:NS}
\left\{
\begin{aligned}
&u_t-\Delta u+u\cdot\nabla u+\nabla \mathrm{p}=0,\quad x\in\mathbb{R}^3,~~t>0,\\
&\nabla\cdot u=0,\\
&u(\cdot,0)=u_0.
\end{aligned}\right.
\end{equation}  
Here, $u$ is the $\mathbb{R}^3$-valued velocity field, $\mathrm{p}$ stands for the scalar pressure, and $u_0$ is an initial datum in the critial space $L^3(\mathbb{R}^3)$. 
The space analyticity of the classical solutions to the Navier-Stokes equations is usually expected as a consequence of parabolic regularity; see, e.g., \cite{kahane1969,giga2002regularizing,bae2012analyticity,xu2020local}. The time analyticity of the solutions to the Navier-Stokes equations can be obtained via analytic semigroup properties and complex variables \cite{foias1989gevrey,giga1983time}. The first general pointwise time analyticity result for the Navier-Stokes equations was obtained recently in \cite{donghongjie2020jfa},  whose proof involves only real variable methods. 
From the Hartogs's theorem \cite{hartogs1906theorie} or Osgood's lemma \cite{osgood1899note}, if a function with several complex variables is analytic with respect to each variable, then it is analytic with respect to all variables. However, for functions with real variables, we do not have such a good property. Hence, combining the space analyticity and time analyticity of solutions to the Navier-Stokes equations does not imply the joint space-time analyticity.  
Moreover, as far as we know, there is no quantitative estimates for the Navier-Stokes equations in the previous literatures which imply the joint space-time analyticity of solutions.
The main purpose of this paper is to provide the quantitative decay rate estimates of the space-time derivatives of solutions to the Navier-Stokes equations that yield the joint space-time analyticity. The main results of this paper are as follows:
\begin{theorem}\label{thm:main}
Let $u_0\in L^{3}(\mathbb{R}^3)$ satisfy $\nabla\cdot u_0=0$, and $u(t)$ be the mild solution (see Definition \ref{def:mild}) to the Navier-Stokes equations \eqref{eq:NS}. Then the following statements hold:

\begin{enumerate}
\item[(i)] There exist $T>0$ and a constant $M>0$ independent of $\beta$ and $k$ and depending on $T$, such that 
\begin{align}\label{eq:analyticinequality}
\left\|D_x^\beta\partial_t^ku(t)\right\|_{L^q(\mathbb{R}^3)}\leq M^{|\beta|+k}\left(|\beta|+k\right)^{|\beta|+k}t^{-\frac{|\beta|}{2}-k-\frac{3}{2}(\frac{1}{3}-\frac{1}{q})}
\end{align}
for $3\leq q\leq\infty$, $ t\in(0,T]$, $\beta\in\mathbb{N}^3$ and $k\in\mathbb{N}$ with $|\beta|+k>0$. 
\item[(ii)] If $\left\|u_0\right\|_{L^3(\mathbb{R}^3)}$ is small enough, the mild solution $u$ exists globally, and there exists a positive constant $M$ independent of $t$, $\beta$ and $k$ such that \eqref{eq:analyticinequality} holds for any time $t>0$.
\end{enumerate}
As a consequence, solution $u$ satisfying \eqref{eq:analyticinequality} is joint space-time analytic for any $t>0.$ 
\end{theorem}
Inequality \eqref{eq:analyticinequality} seems to be the first quantitative estimate for the joint space-time derivatives of solutions to the Navier-Stokes equations, which yields the joint space-time analyticity. Notice that quantitative estimates on the space-time analyticity of solutions are important in the applications to the null-controllability of parabolic evolutions over measurable sets; see, e.g., \cite{zhang2015observation,canzhang2017analyticity}.
Comparing with the previous results about analyticity of solutions to the Navier-Stokes equations (see the next paragraph below), the above estimate shows some novelties: (i) the quantitative joint space-time decay rate estimates and analyticity are obtained; (ii) since the constant $M$ is independent of time for small initial data, the bounds on the growth (in time) of radius of space analyticity, time analyticity, and joint space-time analyticity of solutions are also obtained (see Remark \ref{rmk:radius}).
Moreover, notice that the decay rate for a solution to the heat equation in $\mathbb{R}^3$ (see \eqref{eq:estimate3}) is given by:
\begin{align}\label{eq:heatequationjoint}
\left\|D_x^\beta\partial_t^k\left[G(\cdot,t)\ast f\right]\right\|_{L^q}\leq C^{\frac{|\beta|}{2}+k}\left(|\beta|+k\right)^{\frac{|\beta|}{2}+k+\frac{3}{2}(\frac{1}{p}-\frac{1}{q})}t^{-\frac{|\beta|}{2}-k-\frac{3}{2}(\frac{1}{p}-\frac{1}{q})}\|f\|_{L^p},
\end{align}
where $f\in L^p(\mathbb{R}^3)$, $3\leq p\leq q\leq\infty$, $\beta\in\mathbb{N}^3$, $k\in\mathbb{N}$, and $C$ is a constant independent of $\beta,k$ and $t$. Comparing with \eqref{eq:heatequationjoint}, the decay rate in \eqref{eq:analyticinequality} is almost optimal, except for some possibility of improvement for the index on $(|\beta|+k)$. Note that the joint space-time analyticity also implies the unique continuation property of the mild solutions.

The proof of Theorem \ref{thm:main} is based on some elementary estimates for the space-time derivatives of the heat kernel (see Lemma \ref{lmm:heatkernelestimate2}). Instead of proving \eqref{eq:analyticinequality} directly, we will apply a technique from a recent paper \cite{donghongjie2020jfa}, and change the position of the decay rate $t^k$ (corresponding to time derivatives) from the right hand side of \eqref{eq:analyticinequality} to the left to show 
\begin{align}\label{eq:pretk}
\left\|D_x^\beta\partial_t^k\left(t^ku(t)\right)\right\|_{L^{q}}\leq \hat{M}^{|\beta|+k-\delta}(|\beta|+k)^{|\beta|+k-1}t^{-\frac{|\beta|}{2}-\frac{3}{2}\left(\frac{1}{3}-\frac{1}{q}\right)},\quad 3\leq q\leq\infty
\end{align}
for some $0<\delta<1$. In \cite{donghongjie2020jfa}, Dong and Zhang proved the time analyticity of solutions to the Navier-Stokes equations with the assumption $u\in L^\infty\left(\mathbb{R}^d\times [0,1]\right)$ $(d\in\mathbb{N})$ (see \cite[Theore 3.1]{donghongjie2020jfa}). They obtained (see  \cite[Proposition 3.4]{donghongjie2020jfa})
\begin{align*}
\left\|\partial^k_t(t^ku(t))\right\|_{L^\infty}\leq N^{k-\frac{2}{3}}k^{k-\frac{2}{3}}+C\int_0^t(t-s)^{-\frac{1}{2}}\left\|\partial^k_t(s^ku(s))\right\|_{L^\infty}\di s
\end{align*}
for some constant $C$ depending only on $d$, and some sufficiently large constant $N$ depending on $d$ and $\|u\|_{L^\infty}$, but independent of $k$. After one step of iteration, the above inequality becomes the Gronwall type inequality, which gives 
\[
\sup_{t\in(0,1]}\left\|\partial^k_t(t^ku(t))\right\|_{L^\infty}\leq N^{k-1/2}k^{k-2/3}.
\] 
This implies the time analyticity. Notice that more regularity assumptions for the initial data are essential to obtain the boundedness of solutions, i.e., $u\in L^\infty\left(\mathbb{R}^d\times [0,1]\right)$. In this paper, we do not assume $u\in L^\infty\left(\mathbb{R}^3\times[0,T]\right)$, and the initial data are only required in $L^3(\mathbb{R}^3)$. In this case, we could obtain the following inequality:
\begin{align*}
\left\|\partial^k_t(t^ku(t))\right\|_{L^q}\leq M^{k-\delta}k^{k-1}t^{\frac{3}{2}\left(\frac{1}{3}-\frac{1}{q}\right)}+C\theta\int_0^t(t-s)^{-\frac{1}{2}-\frac{3}{2}\left(1-\frac{1}{a}\right)}s^{-\frac{3}{2}\left(-1+\frac{1}{a}+\frac{1}{3}\right)}\left\|\partial_s^k\left(s^ku(s)\right)\right\|_{L^q}\di s,
\end{align*}
where $\theta$ is a constant depends on initial datum and the local existing time $T$ (see Theorem  \ref{thm:contraction} and Remark \ref{rmk:theta} for details).
The above inequality cannot imply the boundedness of $\left\|\partial^k_t(t^ku(t))\right\|_{L^q}$ for $3\leq q<\infty$ from the Gr\"onwall type inequality. We will use the smallness of $\theta$ to overcome this difficulty (see the proof of Proposition \ref{pro:two}). To prove  \eqref{eq:pretk}, we will only use induction for $|\beta|+k$ and a bootstrapping method without any contraction argument.  One of the difficulties lies in keeping the coefficient $\hat{M}$ invariant as $|\beta|+k$ increasing. In order to overcome this difficulty, we use the property of functions in the following form:
\begin{align*}
f(\hat{M}):=\sup_{|\beta|+k\geq 1}\left(\frac{c_1}{\hat{M}}\right)^{|\beta|+k}(|\beta|+k)^{c_2}=\sup_{|\beta|+k\geq 1}\left(\frac{c_1(|\beta|+k)^{\frac{c_2}{|\beta|+k}}}{\hat{M}}\right)^{|\beta|+k},
\end{align*}
where $c_1$, $c_2>0$ are constants. Since $(|\beta|+k)^{\frac{c_2}{|\beta|+k}}$ is bounded, we see that
$f(\hat{M})\to 0$ as $\hat{M}\to+\infty$. The calculation in the proof of \eqref{eq:pretk} shows that (see Proposition \ref{pro:one} and Proposition \ref{pro:two})
\begin{align*}
\left\|D_x^\beta\partial_t^k\left(t^ku(t)\right)\right\|_{L^{q}}\leq& \left(\frac{c_1}{\hat{M}}\right)^{|\beta|+k}(|\beta|+k)^{c_2}\hat{M}^{|\beta|+k-\delta}(|\beta|+k)^{|\beta|+k-1}t^{-\frac{|\beta|}{2}-\frac{3}{2}\left(\frac{1}{3}-\frac{1}{q}\right)}\\
\leq& f(\hat{M})\hat{M}^{|\beta|+k-\delta}(|\beta|+k)^{|\beta|+k-1}t^{-\frac{|\beta|}{2}-\frac{3}{2}\left(\frac{1}{3}-\frac{1}{q}\right)}.
\end{align*}
By the property of $f(\hat{M})$, we can find a constant $\hat{M}$ big enough to make $f(\hat{M})<1$, which leads to \eqref{eq:pretk} for any $\beta$ and $k$.

The analyticity and Gevery-class regularity of solutions to the Navier-Stokes equations have been studied for several decades. Without the decay rate estimates, analytic semigroup method was used by Giga in \cite{giga1983time} to prove the time analyticity and space analyticity of the weak solutions to the Navier-Stokes equations with zero-boundry condition in a bounded domain of $\mathbb{R}^n$ $(n\geq 2)$.  In \cite{foias1989gevrey}, Foias and Temam provided a method by using Fourier analysis to show the time analyticity of solutions to the Navier-Stokes equations with space periodicity  boundary condition  in a Gevrey class of functions (for the space variable) in space $\mathbb{R}^2$ and $\mathbb{R}^3$.
The Fourier splitting method was introduced by Schobek  in \cite{schonbek1995large} to obtain the decay rate estimates of the homogeneous $H^m(\mathbb{R}^2)$ norms for solutions to the Navier-Stokes equations in $\mathbb{R}^2$ with initial data in $H^m\cap L^1(\mathbb{R}^2)$ ($m\geq3$), and the analyticity of solutions was not proved. Her method was generalized to higher dimensional cases \cite{schonbek1996onthe,oliver2000remark}. Especially, Oliver and Titi \cite{oliver2000remark} used the method based on the Gevrey estimates to present upper bounds for the decay rate of higher order derivatives of solutions to the Navier-Stokes equations in $\mathbb{R}^n$ $(n\geq 1)$:
\begin{align}\label{eq:spacerate}
\|(-\Delta)^{\frac{m}{2}}u(t)\|_{L^2(\mathbb{R}^n)}^2\leq C\left(\frac{2m}{e}\right)^{2m}(1+t)^{-\gamma-m},
\end{align}
Where $\gamma>0$ and $m>0$ are two real numbers. They obtain the above estimate under the conditions $\|u(t)\|_{L^2}^2\leq M/(1+t)^\gamma$ and $\liminf_{t\to \infty}\|u(t)\|_{H^r}<\infty$ for some constants $M>0$ and $r>n/2$. The above decay rate estimate yields explicit bounds on the growth of the radius of space analyticity of the solution in time.
Based on some contraction arguments (or  Gronwall type estimates), space analyticity \cite{kahane1969, giga2002regularizing} and time analyticity \cite{donghongjie2020jfa} for the solutions to the Navier-Stokes equations were obtained. Giga and Sawada \cite{giga2002regularizing} obtained the following decay rate estimates for the space derivatives (see \cite[Theorem 1.1]{giga2002regularizing}):
\begin{align}\label{eq:spacerate1}
\left\|D_x^\beta u(t)\right\|_{L^q(\mathbb{R}^n)}\leq K_1\left(K_2|\beta|\right)^{|\beta|}t^{-\frac{|\beta|}{2}-\frac{n}{2}\left(\frac{1}{n}-\frac{1}{q}\right)},
\end{align}
where the constants $K_1$ and $K_2$ are independent of $\beta$, and $n\leq q\leq \infty$. The space analyticity follows from the above estimates.
Under the boundedness condition  $|u(x,t)|\leq C$ for $(x,t)\in\mathbb{R}^d\times[0,1]$, Dong and Zhang obtained   (see \cite[Theorem 3.1]{donghongjie2020jfa}):
\[
\sup_{t\in (0,1]}t^k\|\partial_t^ku(t)\|_{L^\infty}\leq N^kk^k
\] 
for any $k\in\mathbb{N}$ and some large constant $N$ independent of $k$ (essentially depending on time). The above estimates yield the time analyticity without the bounds on the growth of radius of time analyticity.
One can also find space analyticity results in \cite{bae2012analyticity} for initial data in critical Besov space in $\dot{B}_{p,q}^{\frac{3}{p}-1}(\mathbb{R}^3)$ with $1\leq p<\infty$ and $1\leq q\leq\infty$, and in \cite{xu2020local}  with initial data in $BMO(\mathbb{R}^n)$ $(n\geq2)$. 

We note here that from the Navier-Stokes equations, taking time derivative $\partial_t^k$ ($k$ times) of the solution $u$ corresponds to taking the space derivative $\Delta^k$ ($2k$ times) of $u$. If we change the time derivatives into space derivatives, combining the estimate for space derivatives \eqref{eq:spacerate} or \eqref{eq:spacerate1} yields the following joint space-time Gevery type estimate of solutions to the Navier-Stokes equation \eqref{eq:NS}:
\begin{align*}
\left\|D_x^\beta\partial_t^ku(t)\right\|_{L^q(\mathbb{R}^3)}\leq M^{|\beta|+2k}\left(|\beta|+2k\right)^{|\beta|+2k}t^{-\frac{|\beta|}{2}-k-\frac{3}{2}(\frac{1}{3}-\frac{1}{q})}.
\end{align*}
Although the above Gevery type estimate shows the decay estimates for the $L^q(\mathbb{R}^3)$ ($3\leq q\leq\infty$) norm of space-time derivatives of the solution, the joint space-time analyticity cannot be obtained from it. 

Next, we show the outline of the proof of Theorem \ref{thm:main}:

(1) Instead of proving inequality \eqref{eq:analyticinequality} directly, we will focus on proving \eqref{eq:pretk}, which yields \eqref{eq:analyticinequality} (see Theorem \ref{thm:analytic}).

(2) We will first use induction for $|\beta|+k$ to prove inequality \eqref{eq:pretk} with $3\leq q<\infty$:
\begin{enumerate}
\item[(a)] The first step for induction, i.e., the case for $|\beta|+k=1$, can be verified directly from the regularity estimate of heat kernel; see \eqref{eq:estimate3}.
\item[(b)] Assume \eqref{eq:pretk} holds for $|\beta|+k=L-1$, $L\geq 2$, $3\leq p<+\infty$ and some constant $M$, and we are going to prove inequality \eqref{eq:pretk} for $|\beta|+k=L$ and $3\leq q<+\infty$. Direct calculation from the mild solution (see \eqref{eq:mild}) shows
\begin{equation}\label{eq:esti}
\begin{aligned}
\left\|D_x^\beta\partial_t^k\left(t^ku(t)\right)\right\|_{L^q}\leq& \left\|D_x^\beta\partial_t^k\left[t^kG(\cdot,t)\ast u_0\right]\right\|_{L^q}\\
&+\left\|D_x^\beta\partial_t^k\left[t^k\int_0^t\nabla G(\cdot,t-s)\ast\left[\mathcal{P}(u\otimes u)(s)\right]\di s\right]\right\|_{L^q}.
\end{aligned}
\end{equation}
Here, $\mathcal{P}$ is the Helmholtz projection. For the first term in the right hand side of the above inequality, we can apply the space-time estimates for the heat kernel (see \eqref{eq:estimate3}) to obatin 
\begin{align*}
\left\|D_x^\beta\partial_t^k\left[t^kG(\cdot,t)\ast u_0\right]\right\|_{L^q}\leq  h_1(M)M^{|\beta|+k-\delta}(|\beta|+k)^{|\beta|+k-1}t^{-\frac{|\beta|}{2}-\frac{3}{2}\left(\frac{1}{3}-\frac{1}{q}\right)}
\end{align*}
for any $|\beta|+k>0$ and $3\leq q\leq \infty$. Here, $h_1(M)\to 0$ as $M\to\infty$ (see \eqref{firstterm} for details). 
The estimate of the second term is the most difficult part of the whole proof, which needs more careful and detailed calculations. It will be separated into two cases: $|\beta|>0$ and $|\beta|=0$. For both of these two cases, the target is to obtain the following inequality
\[
\sup_{0<s\leq t}{s^{\frac{|\beta|}{2}+\frac{3}{2}\left(\frac{1}{3}-\frac{1}{q}\right)}\left\|D_x^\beta\partial_s^k\left(s^ku(s)\right)\right\|_{L^q}}\leq h(M) M^{|\beta|+k-\delta}(|\beta|+k)^{|\beta|+k-1}
\]
for some function $h(M)$ satisfying $h(M)\to 0$ as $M\to\infty$, which implies \eqref{eq:pretk} for $3\leq q<\infty$, $|\beta|+k=L$ and $|\beta|>0$. The proof relies on Young's inequality, the estimates of heat kernel (see Lemma \ref{lmm:heatkernelestimate2}) and the fact that the Helmholtz operator $\mathcal{P}$ is a bounded operator from $L^p(\mathbb{R}^3)$ to $L^p(\mathbb{R}^3)$ for $1<p<\infty$ (see Proposition \ref{pro:one} for $|\beta|>0$ and Proposition \ref{pro:two} for $|\beta|=0$ for more details).
Notice that we can obtain the space analyticity of solutions from the case $|\beta|>0$, and the case $|\beta|=0$ corresponds to the time analyticity, which needs some restrictions for the existing time $T$ or the smallness of $\|u_0\|_{L^3}$ (see Remark \ref{rmk:theta}).
\end{enumerate}

(3) Finally, the case for $q=+\infty$ and $|\beta|+k=L$ can be proved by a bootstrapping argument, i.e., the results for $3\leq q<+\infty$ implies the result for $q=\infty$ (see Proposition \ref{pro:three}).

The rest of this paper is organized as follows. In Section \ref{sec:pre}, we will show some elementary space-time estimates for the heat kernel. The well-posedness and some useful estimates of mild solution of the Navier-Stokes equations will also be obtained in this section. In Section \ref{sec:analytic}, we will give the proof of the main result of Theorem \ref{thm:main}. Some useful results will be provided in Appendix \ref{app}.

\section{Preliminaries}\label{sec:pre}
In this section, we will present some useful lemmas. 
Let $G$ be the heat kernel in $\mathbb{R}^3$ given by 
\begin{align*}
	G(x,t)=\frac{1}{(4\pi t)^{\frac{3}{2}}}e^{-\frac{|x|^2}{4t}},\quad x\in\mathbb{R}^3,~t>0.
\end{align*}
We have the following estimates for the space-time derivatives of heat kernel:
\begin{lemma}\label{lmm:heatkernelestimate2}
Let $f\in L^p(\mathbb{R}^3)$, $k\in\mathbb{N}$, $\beta \in\mathbb{N}^3$ and $|\beta|+k>0$. Assume $m\in\mathbb{N}$ satisfying $0\leq m\leq k$. Then, there exists a constant $M_0>0$  independent of $\beta$ and $k$ such that the following inequalities hold for $3\leq q\leq +\infty$:
\begin{align}\label{eq:estimate1}
\left\|D_x^\beta\partial_t^k\left[t^mG(\cdot,t)\right]\right\|_{L^q}\leq M_0^{\frac{|\beta|}{2}+k}\left(|\beta|+k\right)^{\frac{|\beta|}{2}+k+\frac{3}{2}(1-\frac{1}{q})}t^{-\frac{|\beta|+2(k-m)}{2}-\frac{3}{2}(1-\frac{1}{q})},
\end{align}
\begin{align}\label{eq:estimate2}
\left\|D_x^\beta\partial_t^k\left[t^m\nabla G(\cdot,t)\right]\right\|_{L^q}\leq M_0^{\frac{|\beta|}{2}+k}\left(|\beta|+k\right)^{\frac{|\beta|+1}{2}+k+\frac{3}{2}(1-\frac{1}{q})}t^{-\frac{|\beta|+2(k-m)+1}{2}-\frac{3}{2}(1-\frac{1}{q})},
\end{align}
and for $1\leq p\leq q\leq\infty$, we have
\begin{align}\label{eq:estimate3}
\left\|D_x^\beta\partial_t^k\left[t^mG(\cdot,t)\ast f\right]\right\|_{L^q}\leq M_0^{\frac{|\beta|}{2}+k}\left(|\beta|+k\right)^{\frac{|\beta|}{2}+k+\frac{3}{2}(\frac{1}{p}-\frac{1}{q})}t^{-\frac{|\beta|+2(k-m)}{2}-\frac{3}{2}(\frac{1}{p}-\frac{1}{q})}\|f\|_{L^p},
\end{align}
\begin{align}\label{eq:estimate4}
\left\|D_x^\beta\partial_t^k\left[t^m\nabla G(\cdot,t)\ast f\right]\right\|_{L^q}\leq M_0^{\frac{|\beta|}{2}+k}\left(|\beta|+k\right)^{\frac{|\beta|+1}{2}+k+\frac{3}{2}(\frac{1}{p}-\frac{1}{q})}t^{-\frac{|\beta|+2(k-m)+1}{2}-\frac{3}{2}(\frac{1}{p}-\frac{1}{q})}\|f\|_{L^p}. 
\end{align}
Moreover,
\begin{align}\label{eq:key}
\lim_{t\to0}t^{\frac{3}{2}(\frac{1}{p}-\frac{1}{q})}\|G(\cdot,t)\ast f\|_{L^q}=0,\quad \forall ~q>p.
\end{align}
In particularly, for $p=3$ and $q=6$, we obtain from \eqref{eq:key} that
\begin{align}\label{eq:key1}
\lim_{t\to0}t^{\frac{1}{4}}\|G(\cdot,t)\ast f\|_{L^{6}}=0.
\end{align}
\end{lemma}
One can refer to \cite[Proposition 2.1]{wang2021optimal} for details of the proof. When come to the space derivatives with $k=0$, some similar results also can be found in \cite{carrillo2008asymptotic,dong2008spatial,gao2020global}.

Next, we are going to present some results about local well-posedness and regularity of solution to the Navier-Stokes equations with initial data $u_0\in L^3(\mathbb{R}^3)$ satisfying $\nabla\cdot u_0=0$. The mild solutions will be studied in the following space:
\begin{align}\label{eq:spaces}
X_T:=\left\{f\in C_b\left((0,T];L^{3}(\mathbb{R}^3)\right),~~\sup_{t\in(0,T]}t^{\frac{1}{4}}\|f(t)\|_{L^{6}}<\infty\right\}
\end{align}
with norm 
\[
\|f\|_{X_T}:=\max\left\{\sup_{t\in(0,T]}\|f(t)\|_{L^3},\quad\sup_{t\in(0,T]}t^{\frac{1}{4}}\|f(t)\|_{L^{6}} \right\}.
\]
Then, space $(X_T,\|\cdot\|_{X_T})$ is a Banach space. According to Hodge's decomposition, every vector field $v\in L^3(\mathbb{R}^3)$ has a unique orthogonal decomposition:
\begin{align*}
v=w+\nabla g,\quad \nabla\cdot w=0.
\end{align*}
with $w, \nabla g\in L^3\left(\mathbb{R}^3\right)$. Let $\mathcal{P}$ be the Helmholtz projection in $\mathbb{R}^3$. Then, for the solution up to the Navier-Stokes equation we have $\mathcal{P}u=u$ and $\mathcal{P}\nabla p=0$. Apply $\mathcal{P}$ on the first equation of \eqref{eq:NS}, which project the Navier-Stokes equation on the sapce of divergence-free vector fields. The mild solution are defined as follows:
\begin{definition}[Mild solutions]\label{def:mild}
Let  $u_0\in L^3(\mathbb{R}^3)$ satisfy $\nabla\cdot u_0=0. $We call $u\in X_T$ a mild solution to the Navier-Stokes equations \eqref{eq:NS} with initial datum $u_0$ if $u$ satisfies the following Duhamel integral equation in  $X_T$:
\begin{align}\label{eq:mild}
u(t)=G(\cdot,t)\ast u_0-\int_0^t\nabla G(\cdot,t-s)\ast\left[\mathcal{P}(u\otimes u)(s)\right]\di s,\quad t\in[0,T].
\end{align}
If equation \eqref{eq:mild} holds for any $T>0$, then we call $u$ a global mild solution.
\end{definition}
We have the following theorem:
\begin{theorem}\label{thm:contraction}
Let $u_0\in L^{3}(\mathbb{R}^3)$ and $\nabla\cdot u_0=0$. There exists  a constant $\theta>0$ small enough such that if $T>0$ satisfies
\begin{align}\label{eq:important}
\sup_{t\in(0,T]}t^{\frac{1}{4}}\|G(\cdot,t)\ast u_0\|_{L^{6}}\leq \theta,
\end{align}
then there is a unique mild solution $u\in X_{T}$ to the Navier-Stokes equations  \eqref{eq:NS} in the following set:
\[
X_{T}^\theta:=\left\{u\in X_{T}:~~\sup_{t\in(0,T]}t^{\frac{1}{4}}\|u(t)\|_{L^{6}}\leq 2\theta\right\}.
\]
We also have $u\in C\left([0,T]; L^{3}(\mathbb{R}^3)\right)$ and $u(x,0)=u_0(x)$ for $ x\in\mathbb{R}^3$. 

(Small initial data) If the total mass $\|u_0\|_{L^3}$ is small enough, then from \eqref{eq:estimate3} we obtain
\begin{align}\label{eq:assumptiona0}
\sup_{t\in(0,T]}t^{\frac{3}{2}\left(\frac{1}{3}-\frac{1}{q}\right)}\left\|G(\cdot,t)\ast u_0\right\|_{L^q}\leq C\|u_0\|_{L^3}\leq\theta,\qquad 3\leq q\leq\infty,
\end{align}
for any $T>0$, and hence, there is a unique global mild solution.

(Decay rate estimates for derivatives) Moreover, the mild solution $u(t)$  satisfies
\begin{align}\label{eq:regularity}
\left\|D_x^\beta\partial_t^ku(t)\right\|_{L^q(\mathbb{R}^3)}\leq Ct^{-\frac{|\beta|}{2}-k-\frac{3}{2}(\frac{1}{3}-\frac{1}{q})},\quad  3\leq q\leq\infty,
\end{align}
for any $k\in\mathbb{N}$ and $\beta\in\mathbb{N}^3$, and $C>0$ is a constant depends on $\theta, ~\beta$ and $k$.
\end{theorem}
The above local well-posedness results can be directly obtained by the classical Kato's method \cite{kato1984strongl} (see also \cite{weissler1980local,biler1995cauchy,biler2010blowup,carrillo2008asymptotic} for the local well-posedness of mild solutions to other dissipative equations).  When $k=0$, inequalities similar to \eqref{eq:regularity} were obtained for the Navier-Stokes equations \cite{dong2009optimal,giga2002regularizing,sawada2006on,sawada2005analyticity} and the quasi-geostrophic equation  \cite{dong2008spatial}. When $k\neq0$, similar space-time regularity estimate can also obtained; see \cite[Inequality (48)]{dong2008spatial} for the quasi-geostrophic equation. 

Next, we provide a more precise and useful estimate similar to \eqref{eq:regularity} with $|\beta|=k=0$.
We have
\begin{theorem}[$L^q(\mathbb{R}^3)$ estimate]\label{thm:regularity1}
Let $u_0\in L^{3}(\mathbb{R}^3)$, $\nabla\cdot u_0=0$ and $\theta$, $T$ satisfy \eqref{eq:important}. Then, the mild solution $u(t)$ obtained by Theorem \ref{thm:contraction} belongs to $L^q(\mathbb{R}^3)$ for any $3\leq q\leq +\infty$, and we have
\begin{align}\label{eq:q}
\|u(t)\|_{L^q(\mathbb{R}^3)}\leq A\theta t^{-\frac{3}{2}(\frac{1}{3}-\frac{1}{q})},\quad  3\leq q\leq\infty,~~0<t\leq T,
\end{align}
where  $A$ is a constant independent of $u_0$, $T$ and $\theta$.
\end{theorem}
\begin{proof}
Let $u_1(t)$ and $u_2(t)$ be defined by
\begin{align}\label{u1u2}
u_1:=G(\cdot,t)\ast u_0,\quad u_2:=-\int_0^t\nabla G(\cdot,t-s)\ast\left[\mathcal{P}(u\otimes u)(s)\right]\di s.
\end{align}
Due to \eqref{eq:assumptiona0}, we have
\begin{align}\label{eq:Lq1}
\|u_1(t)\|_{L^q}\leq \theta t^{-\frac{3}{2}\left(\frac{1}{3}-\frac{1}{q}\right)},\quad 3\leq q\leq \infty,\quad 0< t\leq T.
\end{align}
Next, we deal with the second term $u_2$ and prove \eqref{eq:q}. From \eqref{eq:estimate4}, we have
\begin{align*}
\|u_2(t)\|_{L^q}\leq\int_0^t\|\nabla G(\cdot,t-s)\ast[\mathcal{P}(u\otimes u)(s)]\|_{L^q}\di s\leq C\int_0^t(t-s)^{-\frac{1}{2}-\frac{3}{2}\left(\frac{1}{3}-\frac{1}{q}\right)}\|\mathcal{P}(u\otimes u)(s)\|_{L^3}\di s.
\end{align*}
Because $\mathcal{P}$ is a bounded operator from $L^3(\mathbb{R}^3)$ to itself, we have
\begin{align*}
\|u_2(t)\|_{L^q}&\leq C\int_0^t(t-s)^{-\frac{1}{2}-\frac{3}{2}\left(\frac{1}{3}-\frac{1}{q}\right)}\|u\otimes u\|_{L^3}\di s\\
&\leq C\int_0^t(t-s)^{-\frac{1}{2}-\frac{3}{2}\left(\frac{1}{3}-\frac{1}{q}\right)}s^{-\frac{1}{2}}\left(s^{\frac{1}{4}}\|u(s)\|_{L^6}\right)^2\di s\leq C\theta^2t^{-\frac{3}{2}\left(\frac{1}{3}-\frac{1}{q}\right)}.
\end{align*}

For the case  $q=\infty$, we obtain from Young's inequality that
\begin{align*}
\|u_2(t)\|_{L^\infty}\leq &\int_0^t\|\nabla G(\cdot,t-s)\|_{L^{\frac{4}{3}}}\|\mathcal{P}(u\otimes u)(s)\|_{L^4}\di s\\
\leq& C\int_0^t(t-s)^{-\frac{7}{8}}\|u\otimes u\|_{L^4}\di s\leq C\int_0^t(t-s)^{-\frac{7}{8}}\|u(s)\|_{L^{8}}^2\di s\\
\leq&C\int_0^t(t-s)^{-\frac{7}{8}} s^{-\frac{5}{8}}\di s\cdot \sup_{0< s\leq t} s^{\frac{5}{8}}\|u(s)\|_{L^8}^2\leq C\theta^2 t^{-\frac{1}{2}}.
\end{align*}
Combining \eqref{eq:Lq1}, we obtain \eqref{eq:q}.
\end{proof}

\section{Joint space-time analyticity}\label{sec:analytic}

In this section, we are going to state and prove the main theorem of this paper, which implies the results in Theorem \ref{thm:main}.
We have

\begin{theorem}\label{thm:analytic}
Assume that $u_0\in L^{3}(\mathbb{R}^3)$, $\nabla\cdot u_0=0$ and $\theta$, $T$ satisfy \eqref{eq:important} for  some small $\theta$. Let $u(t)$ be the mild solution to the Navier-Stokes equations \eqref{eq:NS} in $[0,T]$. 
Then there exists a positive constant  $M$ depending on $\theta$ (or $T$), but independent of $\beta$ and $k$ such that 
\begin{align}\label{eq:analytic}
\left\|D_x^\beta\partial_t^ku(t)\right\|_{L^q(\mathbb{R}^3)}\leq M^{|\beta|+k}\left(|\beta|+k\right)^{|\beta|+k}t^{-\frac{|\beta|}{2}-k-\frac{3}{2}(\frac{1}{3}-\frac{1}{q})}
\end{align}
holds for any $3\leq q\leq\infty$, $ t\in(0,T]$, $ \beta\in\mathbb{N}^3$ and $k\in\mathbb{N}$ with $|\beta|+k>0$. 

Moreover, if $\|u_0\|_{L^3(\mathbb{R}^3)}$ is small enough, there exists a positive constant $M$ independent of $T$, $\beta$ and $k$ such that  inequality \eqref{eq:analytic} holds for any $t\in(0,\infty)$.
\end{theorem}
\begin{remark}\label{rmk:theta}
In the proof of Theorem \ref{thm:analytic}, we need  $\theta$ to be small. Hence, the results hold for two cases:
\begin{enumerate}
	\item [(i)] For an arbitrary initial datum $u_0\in L^3(\mathbb{R}^3)$ with $\nabla\cdot u_0=0$, according to \eqref{eq:key1}, we only need to choose  time $T$ small enough to satisfy \eqref{eq:important}. In this case, we can only obtain \eqref{eq:analytic} with some constant $M$ depending on $T$.
	\item [(ii)] When $\|u_0\|_{L^3(\mathbb{R}^3)}$ is small enough, the constant $\theta$ can also be chosen small (see \eqref{eq:assumptiona0}). In this case, we do not need any assumption on time and the constant $M$ in \eqref{eq:analytic} is independent of time.
\end{enumerate}
\end{remark}

Instead of proving \eqref{eq:analytic} directly, we will apply a technique from a recent paper \cite{donghongjie2020jfa}, and
change the position of the decay rate $t^k$ (corresponding to time derivatives) from the right
hand side of  \eqref{eq:analytic} to the left to
show that there exist some constants $M$ (independent of $\beta$ and $k$) and $0<\delta<1$ such that
\begin{align}\label{eq:Linduction}
\left\|D_x^\beta\partial_t^k\left(t^ku(t)\right)\right\|_{L^{q}}\leq M^{|\beta|+k-\delta}(|\beta|+k)^{|\beta|+k-1}t^{-\frac{|\beta|}{2}-\frac{3}{2}\left(\frac{1}{3}-\frac{1}{q}\right)}
\end{align}
for any $3\leq q\leq\infty$, $t\in(0,T]$. We will separate the proof of \eqref{eq:Linduction} into two cases: $3\leq q<\infty$ and $q=\infty$. 

For the case $3\leq q<\infty$ of inequality \eqref{eq:Linduction}, we are going to prove by induction. First, for $|\beta|+k=1$, we can directly deduce \eqref{eq:Linduction} from \eqref{eq:regularity} for some constant $M>0$. Then, assume that there exists a constant $M$ (to be fixed) independent of $\beta$ and $k$ such that \eqref{eq:Linduction} holds for $3\leq q<\infty$ and $0<|\beta|+k\leq L-1$ for some $L\geq 2$. With the above assumption, we are going to prove that \eqref{eq:Linduction} also holds for $|\beta|+k=L$ with the same $M$. The rest of the induction will be divided into two propositions:
\begin{enumerate}
\item [$\bullet$]
In Proposition \ref{pro:one}, we complete the part of $|\beta|>0$, $|\beta|+k=L$, and $3\leq q<\infty$ in the induction.
\item [$\bullet$] 
In Proposition \ref{pro:two}, we complete the part of $|\beta|=0$, $k=L$, and $3\leq q<\infty$ in the induction, and we finish the proof of \eqref{eq:Linduction} for $3\leq q<\infty$ and $|\beta|+k>0$.
\end{enumerate}
To finish the proof of \eqref{eq:Linduction} with $q=\infty$, we only need to use the result for $3\leq q<\infty$ and some bootstrapping arguments, which will be established in Proposition \ref{pro:three} below.

For convenience, we denote
\begin{align}\label{eq:mu}
\mu_q:=\frac{|\beta|}{2}+\frac{3}{2}\left(\frac{1}{3}-\frac{1}{q}\right)
\end{align}
for $\beta\in\mathbb{N}^3$, $3\leq q\leq\infty$ and $\frac{1}{q}=0$ for $q=\infty$. Let $A_0$ be a constant independent of $p$ such that
\begin{align}\label{eq:A0}
\left\|\nabla G(\cdot,t-s)\right\|_{L^p}\leq A_0(t-s)^{-\frac{1}{2}-\frac{3}{2}\left(1-\frac{1}{p}\right)},\qquad 1\leq p\leq\infty.
\end{align}
We will use $A_1$ to denote the constant generated by the Helmholtz projection $\mathcal{P}$.

\begin{proposition}[$3\leq q<\infty$ and $|\beta|>0$]\label{pro:one}
Let $u_0$, $\theta$ and $T$ satisfy the conditions in Theorem \ref{thm:analytic}. There exists $M>0$ independent of $\beta$, $k$ and $T$ such that if \eqref{eq:Linduction}, i.e.,
\begin{align*}
\left\|D_x^\beta\partial_t^k\left(t^ku(t)\right)\right\|_{L^{q}}\leq M^{|\beta|+k-\delta}(|\beta|+k)^{|\beta|+k-1}t^{-\frac{|\beta|}{2}-\frac{3}{2}\left(\frac{1}{3}-\frac{1}{q}\right)},\quad 3\leq q\leq\infty
\end{align*}
holds for $0<|\beta|+k\leq L-1$ for some $L\geq 2$,
then the same inequality also holds for $|\beta|>0$, $|\beta|+k=L$, $3\leq q<\infty$ and $t\in(0,T]$.
\end{proposition}
\begin{proof}
Direct calculation shows that
\begin{align}\label{eq:induction}
\left\|D_x^\beta\partial_t^k\left(t^ku(t)\right)\right\|_{L^q}\leq \left\|D_x^\beta\partial_t^k\left[t^kG(\cdot,t)\ast u_0\right]\right\|_{L^q}
+\left\|D_x^\beta\partial_t^k\left[t^k\int_0^t\nabla G(\cdot,t-s)\ast\left[\mathcal{P}(u\otimes u)(s)\right]\di s\right]\right\|_{L^q}.
\end{align}
From inequality \eqref{eq:estimate3} with $f=u_0$ and $m=k$, the first term  in \eqref{eq:induction} becomes
\begin{equation*}
\begin{aligned}
\left\|D_x^\beta\partial_t^k\left[t^kG(\cdot,t)\ast u_0\right]\right\|_{L^q}\leq M_0^{|\beta|+k}(|\beta|+k)^{|\beta|+k+\frac{3}{2}(\frac{1}{3}-\frac{1}{q})}\|u_0\|_{L^3}t^{-\mu_q}.
\end{aligned}
\end{equation*}
When $M>M_0$, we have $M^\delta\left(\frac{M_0}{M}\right)^{|\beta|+k}(|\beta|+k)^{1+\frac{3}{2}\left(\frac{1}{3}-\frac{1}{q}\right)}\to 0~\textrm{ as }~|\beta|+k\to \infty.$
Define
\begin{align}\label{eq:h1M}
h_1(M):=\|u_0\|_{L^3}\sup_{|\beta|+k\geq 1}\left[M^\delta\left(\frac{M_0}{M}\right)^{|\beta|+k}(|\beta|+k)^{1+\frac{3}{2}}\right],
\end{align}
and then $\lim_{M\to\infty}h_1(M)=0$. Moreover, we have
\begin{equation}\label{firstterm}
\begin{aligned}
\left\|D_x^\beta\partial_t^k\left[t^kG(\cdot,t)\ast u_0\right]\right\|_{L^q}\leq  h_1(M)M^{|\beta|+k-\delta}(|\beta|+k)^{|\beta|+k-1}t^{-\mu_q}
\end{aligned}
\end{equation}
for any $|\beta|+k>0$ and $3\leq q\leq \infty$.

Next, we estimate the second term in \eqref{eq:induction}.
By the identity $t^k=\sum_{j=0}^k\binom{k}{j}s^{k-j}(t-s)^{j}$, the second term in \eqref{eq:induction} becomes
\begin{equation*}
\begin{aligned}
&\left\|D_x^\beta\partial_t^k\left[t^k\int_0^t\nabla G(\cdot,t-s)\ast\left[\mathcal{P}(u\otimes u)(s)\right]\di s\right]\right\|_{L^q}\\
=&\left\|D_x^\beta\sum_{j=0}^k\binom{k}{j}\partial_t^k\int_0^t\int_{\mathbb{R}^3}(t-s)^j\nabla G(x-y,t-s)s^{k-j}\left[\mathcal{P}(u\otimes u)\right](y,s)\di y\di s\right\|_{L^q}\\
=&\left\|D_x^\beta\sum_{j=0}^k\binom{k}{j}\partial_t^{k-j}\int_0^t\int_{\mathbb{R}^3}\partial_t^j\left[(t-s)^j\nabla G(x-y,t-s)\right]s^{k-j}\left[\mathcal{P}(u\otimes u)\right](y,s)\di y\di s\right\|_{L^q}.
\end{aligned}
\end{equation*}
Changing of variable gives
\begin{equation}\label{shijianjisuan}
\begin{aligned}
&\left\|D_x^\beta\partial_t^k\left[t^k\int_0^t\nabla G(\cdot,t-s)\ast\left[\mathcal{P}(u\otimes u)(s)\right]\di s\right]\right\|_{L^q}\\
=&\left\|D_x^\beta\sum_{j=0}^k\binom{k}{j}\partial_t^{k-j}\int_0^t\int_{\mathbb{R}^3}\partial_s^j\left[s^j\nabla G(x-y,s)\right](t-s)^{k-j}\left[\mathcal{P}(u\otimes u)\right](y,t-s)\di y\di s\right\|_{L^q}\\
=&\left\|D_x^\beta\sum_{j=0}^k\binom{k}{j}\int_0^t\left[\partial_t^j((t-s)^j\nabla G(\cdot,t-s))\right]\ast\left[\partial_s^{k-j}(s^{k-j}\mathcal{P}(u\otimes u)(s))\right]\di s\right\|_{L^q}.
\end{aligned}
\end{equation}
For some $0<\e<1$ to be determined later, we separate the integration into two parts and obtain
\begin{equation}\label{eq:case1}
\begin{aligned}
&\left\|D_x^\beta\partial_t^k\left[t^k\int_0^t\nabla G(\cdot,t-s)\ast\left[\mathcal{P}(u\otimes u)(s)\right]\di s\right]\right\|_{L^q}\\ \leq&\sum_{j=0}^k\binom{k}{j}\int_{(1-\epsilon)t}^t\left\|\left[\partial_t^j((t-s)^j\nabla G(\cdot,t-s))\right]\ast\left[D_x^\beta\partial_s^{k-j}(s^{k-j}\mathcal{P}(u\otimes u)(s))\right]\right\|_{L^q}\di s\\
&+\sum_{j=0}^k\binom{k}{j}\int_0^{(1-\epsilon)t}\left\|\left[D_x^\beta\partial_t^j((t-s)^j\nabla G(\cdot,t-s))\right]\ast\left[\partial_s^{k-j}(s^{k-j}\mathcal{P}(u\otimes u)(s))\right]\right\|_{L^q}\di s\\
=&:I_1+I_2.
\end{aligned}
\end{equation}

\textbf{Estimate of the first term  $I_1$ in \eqref{eq:case1}:} 
From Young's inequality for $1+\frac{1}{q}=\frac{1}{a}+\frac{1}{p}$ and $1\leq a,~p\leq q$, we have
\begin{equation}\label{eq:estI1}
\begin{aligned}
I_1\leq&\sum_{j=1}^k\binom{k}{j}\int_{(1-\epsilon)t}^t\left\|\partial_t^j\left[(t-s)^j\nabla G(\cdot,t-s)\right]\right\|_{L^{a}}\left\|D_x^\beta\partial_s^{k-j}\left[s^{k-j}\mathcal{P}(u\otimes u)(s)\right]\right\|_{L^{p}}\di s\\
&+\int_{(1-\epsilon)t}^t\left\|\nabla G(\cdot,t-s)\right\|_{L^{a}}\left\|D_x^\beta\partial_s^k\left[s^k\mathcal{P}(u\otimes u)(s)\right]\right\|_{L^{p}}\di s=:I_{11}+I_{12}.
\end{aligned}
\end{equation}
By \eqref{eq:estimate2},  there holds
\begin{align}\label{eq:I111}
\left\|\partial_t^j\left[(t-s)^j\nabla G(\cdot,t-s)\right]\right\|_{L^{a}}\leq M_0^{j}j^{j+\frac{1}{2}+\frac{3}{2}(1-\frac{1}{a})}(t-s)^{-\frac{1}{2}-\frac{3}{2}(1-\frac{1}{a})},\quad 1\leq j\leq k.
\end{align}
Using the following inequality
\begin{align}\label{eq:A2}
\left\|D_x^\beta\partial_t^k\left[t^k\mathcal{P}(u\otimes u)(t)\right]\right\|_{L^p}\leq N(1+M^\delta)M^{|\beta|+k-2\delta}\left(|\beta|+k\right)^{|\beta|+k-1}t^{-\frac{|\beta|+1}{2}-\frac{3}{2}\left(\frac{1}{3}-\frac{1}{p}\right)}
\end{align}
for any $3\leq p<\infty$ and $0<|\beta|+k\leq L-1$ (see \eqref{fineestimate1} in Appendix A), we have 
\begin{equation}\label{eq:I112}
\begin{aligned}
&\left\|D_x^\beta\partial_s^{k-j}\left[s^{k-j}\mathcal{P}(u\otimes u)(s)\right]\right\|_{L^{p}}
\leq N(1+M^\delta) M^{|\beta|+k-j-2\delta}(|\beta|+k-j)^{|\beta|+k-j-1}s^{-\frac{|\beta|+1}{2}-\frac{3}{2}(\frac{1}{3}-\frac{1}{p})}
\end{aligned}
\end{equation}
for $1\leq j\leq k$.
Combining \eqref{eq:I111} and \eqref{eq:I112} gives the following estimate for  the term $I_{11}$ in \eqref{eq:estI1}:
\begin{equation}\label{eq:estimateI11}
\begin{aligned}
I_{11}\leq& J_1(\e)\sum_{j=1}^k\binom{k}{j}M_0^{j}j^{j+\frac{1}{2}+\frac{3}{2}(1-\frac{1}{a})}N(1+M^\delta) M^{|\beta|+k-j-2\delta}(|\beta|+k-j)^{|\beta|+k-j-1}t^{-\mu_q}\\
=& N(M^{-\delta}+1)M^{|\beta|+k-\delta}J_1(\e)\sum_{j=1}^k\binom{k}{j}\left(\frac{M_0}{M}\right)^jj^{\frac{3}{2}+\frac{3}{2}(1-\frac{1}{a})}j^{j-1}(|\beta|+k-j)^{|\beta|+k-j-1}t^{-\mu_q},
\end{aligned}
\end{equation}
where $J_1(\e)$ is defined by
\begin{align}\label{eq:J1e}
J_1(\epsilon)=\int_{1-\epsilon}^1(1-s)^{-\frac{1}{2}-\frac{3}{2}\left(1-\frac{1}{a}\right)}s^{-\frac{|\beta|+1}{2}-\frac{3}{2}\left(\frac{1}{3}-\frac{1}{p}\right)}\di s.
\end{align}
Here we choose $1\leq a<\frac{3}{2}$, and then
$-\frac{1}{2}-\frac{3}{2}\left(1-\frac{1}{a}\right)>-1.$
For some constant $K$ to be determined, we set
\begin{equation}\label{eq:ebeta}
\e=\e(|\beta|):=(K|\beta|)^{-1}.
\end{equation}
Then, we have
\begin{equation}
J_1(\e)\leq C(K|\beta|)^{-(\frac{1}{2}-\frac{3}{2}(1-\frac{1}{a}))}\left(1+\frac{1}{K|\beta|-1}\right)^{\frac{|\beta|+1}{2}+\frac{3}{2}\left(\frac{1}{3}-\frac{1}{p}\right)}.
\end{equation}
By the choice of $a$, we have $\frac{1}{2}-\frac{3}{2}(1-\frac{1}{a})>0$ and hence $J_1(\e)\to 0$ as $K|\beta|\to\infty$. Therefore, we fixed $K$ big enough such that
\begin{align}\label{eq:J1}
J_1(\e)\leq J:=\frac{1}{4AA_0A_1\theta}
\end{align}
holds for any $|\beta|>0$.
Here, we choose the above two constants $J$ and $K$ because they will also be used in the following estimate of \eqref{eq:estI12}.
For any $M>M_0$, notice that $\left(\frac{M_0}{M}\right)^jj^{\frac{3}{2}+\frac{3}{2}(1-\frac{1}{a})}\to 0$ as $j\to\infty$. Set
\begin{align}\label{eq:h2M}
h_2(M):= 2\lambda N J \sup_{j\geq 1}\left[\left(\frac{M_0}{M}\right)^jj^{3}\right]
\end{align}
and then $h_2$ satisfies $\lim_{M\to\infty}h_2(M)=0$, where $\lambda$ is the constant in Lemma \ref{lmm:multisequences}.
Therefore, from \eqref{eq:estimateI11} we obtain
\begin{equation}\label{eq:estI11}
\begin{aligned}
I_{11}\leq h_2(M)M^{|\beta|+k-\delta}(|\beta|+k)^{|\beta|+k-1}t^{-\mu_q},
\end{aligned}
\end{equation}
where  we used Lemma \ref{lmm:multisequences}  in the last step.

For the term $I_{12}$ in \eqref{eq:estI1},  we obtain from  the following inequality (See \eqref{fineestimate2} in Appendix A)
\begin{multline}\label{eq:A3}
\left\|D_x^\beta\partial_t^k\left[t^k\mathcal{P}(u\otimes u)(t)\right]\right\|_{L^p}\leq NM^{|\beta|+k-2\delta}\left(|\beta|+k\right)^{|\beta|+k-1}t^{-\frac{|\beta|+1}{2}-\frac{3}{2}\left(\frac{1}{3}-\frac{1}{p}\right)}\\
+A_1\left\|\left[D_x^\beta\partial_t^k\left(t^ku(t)\right)\right]\otimes u(t)\right\|_{L^p}+A_1\left\|u(t)\otimes \left[D_x^\beta\partial_t^k\left(t^ku(t)\right)\right]\right\|_{L^p},
\end{multline}
and H\"older's inequality that
\begin{align*}
&\left\|D_x^\beta\partial_s^k\left[s^k\mathcal{P}(u\otimes u)(s)\right]\right\|_{L^{p}}\leq NM^{|\beta|+k-2\delta}\left(|\beta|+k\right)^{|\beta|+k-1}s^{-\frac{|\beta|+1}{2}-\frac{3}{2}\left(\frac{1}{3}-\frac{1}{p}\right)}\\
&\qquad \qquad +A_1\left\|\left[D_x^\beta\partial_s^k\left(s^ku(s)\right)\right]\otimes u(s)\right\|_{L^{p}}+A_1\left\|u(s)\otimes \left[D_x^\beta\partial_s^k\left(s^ku(s)\right)\right]\right\|_{L^{p}}\\
\leq&NM^{|\beta|+k-2\delta}\left(|\beta|+k\right)^{|\beta|+k-1}s^{-\frac{|\beta|+1}{2}-\frac{3}{2}\left(\frac{1}{3}-\frac{1}{p}\right)}+A_1\left\|D_x^\beta\partial_s^k(s^ku(s))\right\|_{L^q}\|u( s)\|_{L^{\frac{pq}{q-p}}}.
\end{align*}
By \eqref{eq:A0},  we have the following estimate for the term $I_{12}$ in \eqref{eq:estI1}:
\begin{equation}\label{eq:estm}
\begin{aligned}
I_{12}\leq& A_0NM^{|\beta|+k-2\delta}\left(|\beta|+k\right)^{|\beta|+k-1}J_1(\epsilon)t^{-\frac{|\beta|}{2}-\frac{3}{2}\left(\frac{1}{3}-\frac{1}{q}\right)}\\
&+2A_0A_1\int_{(1-\epsilon)t}^t\left\|D_x^\beta \partial_s^k(s^ku(s)) \right\|_{L^q}\|u( s)\|_{L^{\frac{pq}{q-p}}}(t-s)^{-\frac{1}{2}-\frac{3}{2}\left(1-\frac{1}{a}\right)}\di s\\
\leq&A_0NJ_1(\epsilon)M^{|\beta|+k-2\delta}\left(|\beta|+k\right)^{|\beta|+k-1}t^{-\mu_q}\\
&+2AA_0A_1\theta\int_{(1-\epsilon)t}^t\left\|D_x^\beta\partial_s^k(s^ku(s))\right\|_{L^q}(t-s)^{-\frac{1}{2}-\frac{3}{2}\left(1-\frac{1}{a}\right)}s^{-\frac{3}{2}\left(-1+\frac{1}{a}+\frac{1}{3}\right)}\di s,
\end{aligned}
\end{equation}
Set
\begin{align}\label{eq:phi}
\phi(t):=\sup_{0<s\leq t}{s^{\mu_q}\left\|D_x^\beta\partial_s^k\left(s^ku(s)\right)\right\|_{L^q}}.
\end{align}
Set
\begin{align}\label{eq:h3M}
h_3(M):=A_0NJM^{-\delta},
\end{align}
and by the choice of $J_1(\epsilon)$ in \eqref{eq:J1}, we obtain
\begin{equation}\label{eq:estI12}
\begin{aligned}
I_{12}\leq&h_3(M)M^{|\beta|+k-\delta}\left(|\beta|+k\right)^{|\beta|+k-1}t^{-\mu_q}+2AA_0A_1\theta J_1(\e)\cdot\phi(t)t^{-\mu_q}\\
\leq&h_3(M)M^{|\beta|+k-\delta}\left(|\beta|+k\right)^{|\beta|+k-1}t^{-\mu_q}+\frac{1}{2}\cdot\phi(t)t^{-\mu_q}.
\end{aligned}
\end{equation}
Combining \eqref{eq:estI1}, \eqref{eq:estI11} and \eqref{eq:estI12} gives
\begin{equation}\label{eq:estimateI1}
\begin{aligned}
I_1\leq [h_2(M)+h_3(M)]M^{|\beta|+k-\delta}\left(|\beta|+k\right)^{|\beta|+k-1}t^{-\mu_q}+\frac{1}{2}\phi(t)t^{-\mu_q}.
\end{aligned}
\end{equation}

\textbf{Estimate of the second term $I_2$ in \eqref{eq:case1}:}
Combining Young's convolution inequality for $1+\frac{1}{q}=\frac{1}{a}+\frac{1}{p}$ and $1\leq a,~p\leq q$, \eqref{eq:estimate2} and \eqref{eq:A2}, we have
\begin{equation}\label{I_2}
\begin{aligned}
I_2\leq&\sum_{j=0}^{k}\binom{k}{j}\int_0^{(1-\epsilon)t}\left\|D_x^\beta\partial_t^j\left[(t-s)^j\nabla G(\cdot,t-s)\right]\right\|_{L^{a}}\left\|\partial_s^{k-j}\left[s^{k-j}\mathcal{P}(u\otimes u)(s)\right]\right\|_{L^{p}}\di s\\
\leq&\sum_{j=0}^k\binom{k}{j}M_0^{\frac{|\beta|}{2}+j}(|\beta|+j)^{\frac{|\beta|+1}{2}+j+\frac{3}{2}(1-\frac{1}{a})} N(1+M^\delta)M^{k-j-2\delta}(k-j)^{k-j-1}J_2(\epsilon)t^{-\mu_q},
\end{aligned}
\end{equation}
where 
\begin{align}\label{eq:J2ep}
J_2(\epsilon)=\int_0^{1-\epsilon}(1-s)^{-\frac{|\beta|+1}{2}-\frac{3}{2}\left(1-\frac{1}{a}\right)}s^{-\frac{1}{2}-\frac{3}{2}\left(\frac{1}{3}-\frac{1}{p}\right)}\di s,
\end{align}
and $N$ comes from inequality \eqref{eq:A2}.
Here, notice that
$-\frac{1}{2}-\frac{3}{2}\left(\frac{1}{3}-\frac{1}{p}\right)>-1.$
For $\epsilon=(K|\beta|)^{-1}$, we have $J_2(\epsilon)\leq C(K|\beta|)^{\frac{|\beta|+1}{2}+\frac{3}{2}\left(1-\frac{1}{a}\right)}$
for some constant $C$, which implies
\begin{align*}
I_2\leq CN(M^{-\delta}+1)M^{|\beta|+k-\delta}\sum_{j=0}^k\binom{k}{j}\frac{M_0^{\frac{|\beta|}{2}+j}K^{\frac{|\beta|+1}{2}+\frac{3}{2}(1-\frac{1}{a})}}{M^{|\beta|+j}}(|\beta|+j)^{2+\frac{3}{2}(1-\frac{1}{a})} (|\beta|+j)^{|\beta|+j-1}(k-j)^{k-j-1}t^{-\mu_q}.
\end{align*}
When $M>\max\{M_0,K\}$ and $|\beta|+j\to\infty,$ we have
\[
\frac{M_0^{\frac{|\beta|}{2}+j}K^{\frac{|\beta|+1}{2}+\frac{3}{2}(1-\frac{1}{a})}}{M^{|\beta|+j}}(|\beta|+j)^{2+\frac{3}{2}(1-\frac{1}{a})}\to 0.
\]
Set 
\begin{align}
h_4(M):=\frac{2CN}{\lambda}\sup_{|\beta|+j\geq1}\left[\frac{M_0^{\frac{|\beta|}{2}+j}K^{\frac{|\beta|+1}{2}+\frac{3}{2}}}{M^{|\beta|+j}}(|\beta|+j)^{2+\frac{3}{2}}\right],
\end{align}
and then $h_4$ satisfies $\lim_{M\to\infty}h_4(M)=0$. We have
\begin{equation}\label{eq:I2final}
\begin{aligned}
I_2\leq h_4(M)M^{|\beta|+k-\delta}(|\beta|+k)^{|\beta|+k-1}t^{-\mu_q}. 
\end{aligned}
\end{equation}
Define $h(M):=2[h_1(M)+h_2(M)+h_3(M)+h_4(M)]$.
Then we have $\lim_{M\to\infty}h(M)=0$. Combining \eqref{eq:induction}, \eqref{firstterm}, \eqref{eq:case1}, \eqref{eq:estimateI1}, and \eqref{eq:I2final}, we finally obtain
\begin{align*}
\left\|D_x^\beta\partial_t^k\left(t^ku(t)\right)\right\|_{L^q}\leq\frac{1}{2}h(M)M^{|\beta|+k-\delta}(|\beta|+k)^{|\beta|+k-1}t^{-\mu_q}+\frac{1}{2}\phi(t)t^{-\mu_q},
\end{align*}
which implies $\phi(t)\leq h(M) M^{|\beta|+k-\delta}(|\beta|+k)^{|\beta|+k-1}$.
Hence, there exist $M$ big enough, which is independent of $\beta$ and $k$, such that $h(M)<1$, and inequality \eqref{eq:Linduction} holds for $|\beta|+k=L$ when $|\beta|>0$.
\end{proof}

\begin{proposition}[$3\leq q<\infty$ and $|\beta|=0$]\label{pro:two}
Let $u_0$, $\theta$ and $T$ satisfy the conditions in Theorem \ref{thm:analytic}. There exists $M>0$ independent of $\beta$, $k$ and $T$ such that if \eqref{eq:Linduction}, i.e.,
\begin{align*}
\left\|D_x^\beta\partial_t^k\left(t^ku(t)\right)\right\|_{L^{q}}\leq M^{|\beta|+k-\delta}(|\beta|+k)^{|\beta|+k-1}t^{-\frac{|\beta|}{2}-k-\frac{3}{2}\left(\frac{1}{3}-\frac{1}{q}\right)},\quad 3\leq q\leq\infty
\end{align*}
holds for $0<|\beta|+k\leq L-1$ for some $L\geq 2$, then the same inequality also holds for $|\beta|=0$, $k=L$, $3\leq q<\infty$ and $t\in(0,T]$.
\end{proposition}
\begin{proof}
Recall our assumption for induction: assume that there exists a constant $M$ independent of $k$ such that \eqref{eq:Linduction} holds for $|\beta|=0$ and $0< k\leq L-1$ for some positive integer $L$. We are going to prove \eqref{eq:Linduction} for $k=L$. 

As in Proposition \ref{pro:one}, we will estimate the two terms in \eqref{eq:induction} with $|\beta|=0$. Because \eqref{firstterm} holds for any $|\beta|+k>0$, we only need to deal with the second term in \eqref{eq:induction}.
From  \eqref{shijianjisuan}, we have the following estimate for the second term in \eqref{eq:induction}:
\begin{equation}\label{eq:case2}
\begin{aligned}
&\left\|\partial_t^k\left[t^k\int_0^t\nabla G(\cdot,t-s)\ast\left[\mathcal{P}(u\otimes u)(s)\right]\di s\right]\right\|_{L^q}\\
\leq&  \sum_{j=0}^k\binom{k}{j}\int_0^t\left\|\left[\partial_t^j\left((t-s)^j\nabla G(\cdot,t-s)\right)\right]\ast\left[\partial_s^{k-j}\left(s^{k-j}\mathcal{P}(u\otimes u)(s)\right)\right]\right\|_{L^q}\di s\\
=&
\int_0^t\left\|\nabla G(\cdot,t-s)\ast\left[\partial_s^{k}\left(s^{k}\mathcal{P}(u\otimes u)(s)\right)\right]\right\|_{L^q}\di s\\
&+
\sum_{j=1}^{k-1}\binom{k}{j}\int_0^t\left\|\left[\partial_t^j\left((t-s)^j\nabla G(\cdot,t-s)\right)\right]\ast\left[\partial_s^{k-j}\left(s^{k-j}\mathcal{P}(u\otimes u)(s)\right)\right]\right\|_{L^{q}}\di s\\
&+
\int_0^t\left\|\left[\partial_t^k\left((t-s)^k\nabla G(\cdot,t-s)\right)\right]\ast\left(\mathcal{P}(u\otimes u)(s)\right)\right\|_{L^q}\di s=:S_1+S_2+S_3.
\end{aligned}
\end{equation}
Using Young's inequality with $1+\frac{1}{q}=\frac{1}{a}+\frac{1}{p}$, $1\leq a,~p\leq q$ and identity \eqref{eq:A3} to obtain
\begin{equation*}
\begin{aligned}
S_1\leq &\int_0^t\|\nabla G(\cdot,t-s)\|_{L^a}\left\| \partial_s^k(s^k\mathcal{P}(u\otimes u)(s))\right\|_{L^p}\di s\\
\leq& A_0NM^{k-2\delta}k^{k-1}\int_0^t(t-s)^{-\frac{1}{2}-\frac{3}{2}(1-\frac{1}{a})}s^{-\frac{1}{2}-\frac{3}{2}(\frac{1}{3}-\frac{1}{p})}\di s\\
&+A_0A_1\int_0^t(t-s)^{-\frac{1}{2}-\frac{3}{2}(1-\frac{1}{a})} \Big[\left\| \partial_s^k\left(s^ku(s)\right)\otimes u(s)\right\|_{L^p}+\left\|u(s)\otimes\partial_s^k\left(s^ku(s)\right)\right\|_{L^p}\Big]\di s.
\end{aligned}
\end{equation*}
We choose $1\leq a<\frac{3}{2}$, we have 
$-\frac{1}{2}-\frac{3}{2}\left(1-\frac{1}{a}\right)>-1,  -\frac{1}{2}-\frac{3}{2}\left(\frac{1}{3}-\frac{1}{p}\right)>-1$ and
\begin{equation}\label{eq:SS}
S_1\leq A_0NBM^{k-2\delta}k^{k-1}t^{-\mu_q}+2AA_0A_1\theta\int_0^t(t-s)^{-\frac{1}{2}-\frac{3}{2}\left(1-\frac{1}{a}\right)}s^{-\frac{3}{2}\left(-1+\frac{1}{a}+\frac{1}{3}\right)}\left\|\partial_s^k\left(s^ku(s)\right)\right\|_{L^q}\di s,
\end{equation}
where $\mu_q$ is defined by \eqref{eq:mu} with $|\beta|=0$, and $B:=\int_0^1(1-s)^{-\frac{1}{2}-\frac{3}{2}\left(1-\frac{1}{a}\right)}s^{-\frac{1}{2}-\frac{3}{2}\left(\frac{1}{3}-\frac{1}{p}\right)}\di s$. Define
\begin{align*}
h_5(M):=A_0NBM^{-\delta},
\end{align*}
and then for $\phi$ defined by \eqref{eq:phi}, we have
\begin{equation}\label{eq:S1}
\begin{aligned}
S_1\leq  h_5(M)M^{k-\delta}k^{k-1}t^{-\mu_q}+2AA_0A_1B\theta t^{-\mu_q}\cdot\phi(t).
\end{aligned}
\end{equation}

Next, we estimate the term $S_2$ in \eqref{eq:case2}. By Young's inequality with $1+\frac{1}{q}=\frac{1}{a}+\frac{1}{p}$ and $1\leq a,~p\leq q$, we have
\begin{equation*}
\begin{aligned}
S_2\leq \sum_{j=1}^{k-1}\binom{k}{j}\int_0^{t}\left\|\partial_t^j\left((t-s)^j\nabla G(\cdot,t-s)\right)\right\|_{L^{a}}\left\|\partial_s^{k-j}\left(s^{k-j}\mathcal{P}(u\otimes u)(s)\right)\right\|_{L^{p}}\di s.
\end{aligned}
\end{equation*}
Because $1\leq j\leq k-1$, we  use inequalities \eqref{eq:estimate2} and \eqref{eq:A2} to obtain
\begin{equation}\label{eq:S21}
\begin{aligned}
S_{2}\leq&\sum_{j=1}^{k-1}\binom{k}{j}M_0^jj^{\frac{1}{2}+j+\frac{3}{2}\left(1-\frac{1}{a}\right)}N\left(1+M^\delta\right)M^{k-j-2\delta}(k-j)^{k-j-1}\int_0^{t}(t-s)^{-\frac{1}{2}-\frac{3}{2}\left(1-\frac{1}{a}\right)}s^{-\frac{1}{2}-\frac{3}{2}\left(\frac{1}{3}-\frac{1}{p}\right)}\di s\\
\leq&\left[\sup_{1\leq j\leq k-1}\left(\frac{M_0}{M}\right)^jj^{\frac{3}{2}+\frac{3}{2}\left(1-\frac{1}{a}\right)}\right]\lambda NB\left(M^{-\delta}+1\right)M^{k-\delta}k^{k-1}t^{-\mu_q}.
\end{aligned}
\end{equation}
Set
$
h_6(M):=2\lambda NB\left[\sup_{j\geq1}\left(\frac{M_0}{M}\right)^jj^{3}\right],
$
and then
\begin{align}\label{eq:S2}
S_2 \leq h_6(M)M^{k-\delta}k^{k-1}t^{-\mu_q}.
\end{align}

For  the term $S_3$ in \eqref{eq:case2}, we use Young's inequality again with $1+\frac{1}{q}=\frac{1}{a}+\frac{1}{p}$, $1\leq a,~p\leq q$, and H\"older inequality with $\frac{1}{p}=\frac{1}{p_1}+\frac{1}{p_2}$ to obtain
\begin{align*}
S_3\leq& \int_0^{t}\left\|\left[\partial_t^k\left((t-s)^k\nabla G(\cdot,t-s)\right)\right]\right\|_{L^{a}}\left\|\mathcal{P}(u\otimes u)(s)\right\|_{L^{p}}\di s\\
\leq&A_1\int_0^{t}M_0^kk^{\frac{1}{2}+k+\frac{3}{2}\left(1-\frac{1}{a}\right)}(t-s)^{-\frac{1}{2}-\frac{d}{2}\left(1-\frac{1}{a}\right)}\|u(s)\|_{L^{p_1}}\|u(s)\|_{L^{p_2}}\di s\\
=& A^2A_1B\theta^2M^\delta\left(\frac{M_0}{M}\right)^k k^{\frac{3}{2}+\frac{3}{2}\left(1-\frac{1}{q}\right)}\cdot M^{k-\delta}k^{k-1}t^{-\frac{3}{2}\left(\frac{1}{3}-\frac{1}{q}\right)},
\end{align*}
where $A$ comes from \eqref{eq:q}. Set
\begin{align*}
h_7(M):=A^2A_1B\theta^2\sup_{k\geq1}\left[M^\delta\left(\frac{M_0}{M}\right)^kk^{\frac{3}{2}+\frac{3}{2} }\right],
\end{align*}
and then
\begin{align}\label{eq:S3}
S_3 \leq h_7(M)M^{k-\delta}k^{k-1}t^{-\mu_q}.
\end{align}
Combining \eqref{eq:induction}, \eqref{firstterm}, \eqref{eq:case2}, \eqref{eq:S1}, \eqref{eq:S2}, and \eqref{eq:S3}, we have the following estimate for the second term in \eqref{eq:induction}:
\begin{equation}\label{eq:case22}
\begin{aligned}
&\left\|\partial_t^k(t^k u(t))\right\|_{L^q}
\leq\tilde{h}(M)M^{k-\delta}k^{k-1}t^{-\mu_q}+2AA_0A_1B\theta t^{-\mu_q}\cdot\phi(t).
\end{aligned}
\end{equation}
where $\tilde{h}(M)=h_1(M)+h_5(M)+h_6(M)+h_7(M)$. Moreover, we have $\tilde{h}(M)\to0$ as $M\to\infty$.
Therefore,
\begin{align}
\phi(t)\leq\tilde{h}(M)M^{k-\delta}k^{k-1}+2AA_0A_1B\theta\cdot\phi(t).
\end{align}
When $\theta$ is small enough, we have 
\begin{align*}
\phi(t)\leq\frac{\tilde{h}(M)}{1-2AA_0A_1B\theta}M^{k-\delta}k^{k-1}.
\end{align*}
Hence, there exist $M$ big enough, which is independent of $\beta$ and $k$, such that $\frac{\tilde{h}(M)}{1-2AA_0A_1B\theta}<1$, and inequality \eqref{eq:Linduction} holds for $|\beta|=0$.
\end{proof}

So far, we have completed the induction to prove \eqref{eq:Linduction} for $3\leq q<\infty$.  

\begin{remark}\label{Gronwall}
In \cite{donghongjie2020jfa}, Dong and Zhang proved the time analyticity of solutions to the Navier-Stokes equations with the assumption $u\in L^\infty\left(\mathbb{R}^d\times [0,1]\right)$ $(d\in\mathbb{N})$ (see \cite[Theore 3.1]{donghongjie2020jfa}). They obtained (see  \cite[Proposition 3.4]{donghongjie2020jfa})
\begin{align}\label{eq:dongtime}
\left\|\partial^k_t(t^ku(t))\right\|_{L^\infty}\leq N^{k-\frac{2}{3}}k^{k-\frac{2}{3}}+C\int_0^t(t-s)^{-\frac{1}{2}}\left\|\partial^k_t(s^ku(\cdot,s))\right\|_{L^\infty}\di s
\end{align}
for some constant $C$ depending only on $d$, and some sufficiently large constant $N$ depending on $d$ and $\|u\|_{L^\infty}$, but independent of $k$. After one step of iteration, the above inequality becomes the Gr\"onwall type inequality, which gives 
\[
\sup_{t\in(0,1]}\left\|\partial^k_t(t^ku(t))\right\|_{L^\infty}\leq N^{k-1/2}k^{k-2/3}.
\] 
This implies the time analyticity. Notice that more regularity assumptions for the initial data are essential to obtain the boundedness of solutions, i.e., $u\in L^\infty\left(\mathbb{R}^d\times [0,1]\right)$. In this paper, we do not assume $u\in L^\infty\left(\mathbb{R}^3\times[0,T]\right)$, and the initial data are only required in $L^3(\mathbb{R}^3)$. In this case, we could obtain the following inequality:
\begin{align}\label{eq:G}
\left\|\partial^k_t(t^ku(t))\right\|_{L^q}\leq M^{k-\delta}k^{k-1}t^{\frac{3}{2}\left(\frac{1}{3}-\frac{1}{q}\right)}+C\theta\int_0^t(t-s)^{-\frac{1}{2}-\frac{3}{2}\left(1-\frac{1}{a}\right)}s^{-\frac{3}{2}\left(-1+\frac{1}{a}+\frac{1}{3}\right)}\left\|\partial_s^k\left(s^ku(\cdot,s)\right)\right\|_{L^q}\di s,
\end{align}
where $\theta$ is a constant depends on initial datum and the local existing time $T$ (see Theorem  \ref{thm:contraction} and Remark \ref{rmk:theta} for details).
The above inequality cannot imply the boundedness of $\left\|\partial^k_t(t^ku(t))\right\|_{L^q}$ for $3\leq q<\infty$ from the Gr\"onwall type inequality. We will use the smallness of $\theta$ to overcome this difficulty.
\end{remark}

\begin{proposition}[$q=\infty$]\label{pro:three}
Let $u_0$, $\theta$ and $T$ satisfy the conditions in Theorem \ref{thm:analytic}. There exists $M>0$ independent of $\beta$, $k$ and $T$ such that if \eqref{eq:Linduction},i.e.,
\begin{align*}
\left\|D_x^\beta\partial_t^k\left(t^ku(t)\right)\right\|_{L^{q}}\leq M^{|\beta|+k-\delta}(|\beta|+k)^{|\beta|+k-1}t^{-\frac{|\beta|}{2}-k-\frac{3}{2}\left(\frac{1}{3}-\frac{1}{q}\right)},\quad 3\leq q\leq\infty
\end{align*}
 holds for $0<|\beta|+k\leq L-1$ for some $L\geq 2$, then inequality \eqref{eq:Linduction} holds for $q=\infty$ ,$|\beta|+k=L$ and $t\in(0,T]$.
\end{proposition}
\begin{proof}
We will sketch the main idea. From Proposition \ref{pro:one} and Proposition \ref{pro:two}, there exists $M>0$ such that
\begin{align}\label{eq:Linduction3}
\left\|D_x^\beta\partial_t^k\left(t^k u(t)\right)\right\|_{L^{q}}\leq M^{|\beta|+k}(|\beta|+k)^{|\beta|+k}t^{-\mu_q},\quad 3\leq q<\infty.
\end{align}
Since the constant $M$ might depend on $q$, we can not pass to the limit $q\to\infty$ to obtain the estimate for $L^\infty$ norm.
We are going to use the results for $3\leq q<\infty$ to prove \eqref{eq:Linduction} for $q=\infty$. We only need to estimate the $L^\infty$ norm for the second term in \eqref{eq:induction}. From \eqref{shijianjisuan}, we have
\begin{align*}
&\left\|D_x^\beta\partial_t^k\left[t^k\int_0^t\nabla G(\cdot,t-s)\ast\left[\mathcal{P}(u\otimes u)(s)\right]\di s\right]\right\|_{L^\infty}\\
=&\left\|D_x^\beta\sum_{j=0}^k\binom{k}{j}\int_0^t\left[\partial_t^j((t-s)^j\nabla G(\cdot,t-s))\right]\ast\left[\partial_s^{k-j}(s^{k-j}\mathcal{P}(u\otimes u)(s)\right]\di s\right\|_{L^\infty}\\
\leq&\sum_{j=0}^k\binom{k}{j}\int_0^{t/2}\left\|D_x^\beta\partial_t^j((t-s)^j\nabla G(\cdot,t-s))\right\|_{L^{\frac{3}{2}}}\left\|\partial_s^{k-j}(s^{k-j}\mathcal{P}(u\otimes u)(s))\right\|_{L^3}\di s \\
&+ \sum_{j=0}^k\binom{k}{j}\int_{t/2}^t\left\|\partial_t^j((t-s)^j\nabla G(\cdot,t-s))\right\|_{L^{\frac{3}{2}}}\left\|D_x^\beta\partial_s^{k-j}(s^{k-j}\mathcal{P}(u\otimes u)(s))\right\|_{L^{3}}\di s.
\end{align*}
We only need to use \eqref{eq:estimate2} and an inequality similar to \eqref{eq:A2} to estimate the above two terms separately. Then, we could obtain \eqref{eq:Linduction3} for $q=\infty$ with some constant slightly bigger than $M$.
\end{proof}

Now, we are ready to give the proof of Theorem \ref{thm:analytic}.
\begin{proof}[Proof of theorem \ref{thm:analytic}]
Assume $k>0$. Notice that for $j=0,1,2,\cdots,k-1$, we have
\begin{equation*}
\begin{aligned}
t^jD_x^\beta\partial_t^k\left(t^{k-j}u(t)\right)=&t^j\left[D_x^\beta\sum_{i=0}^k\binom{k}{i}\partial_t^it\partial_t^{k-i}\left(t^{k-j-1}u(t)\right)\right]\\
=&t^j\left[D_x^\beta t\partial_t^k\left(t^{k-j-1}u(t)\right)+kD_x^\beta\partial_t^{k-1}\left(t^{k-j-1}u(t)\right)\right]\\
=&t^{j+1}\left[D_x^\beta\partial_t^k\left(t^{k-j-1}u(t)\right)\right]+kt^j\left[D_x^\beta\partial_t^{k-1}\left(t^{k-j-1}u(t)\right)\right],
\end{aligned}
\end{equation*}
which implies
\begin{align}\label{eq:kinduc}
t^{j+1}\left[D_x^\beta\partial_t^k\left(t^{k-j-1}u(t)\right)\right]=t^jD_x^\beta\partial_t^k\left(t^{k-j}u(t)\right)-kt^j\left[D_x^\beta\partial_t^{k-1}\left(t^{k-j-1}u(t)\right)\right].
\end{align}
Combining $j=0$ in \eqref{eq:kinduc} and inequality \eqref{eq:Linduction}, we have
\begin{align*}
\sup_{t\in(0,T]}\left\|t^{\mu_q+1}D_x^\beta\partial_t^k\left(t^{k-1}u(t)\right)\right\|_{L^q}\leq&\sup_{t\in(0,T]}\left\|t^{\mu_q}D_x^\beta\partial_t^k(t^ku(t))\right\|_{L^q}+\sup_{t\in(0,T]}\left\|kt^{\mu_q}D_x^\beta\partial_t^{k-1}(t^{k-1}u(t))\right\|_{L^q}\\
\leq&M^{|\beta|+k-\delta}\left(1+\frac{1}{M}\right)\left(|\beta|+k\right)^{|\beta|+k-1}.
\end{align*}
Similarly, for any $0<\tilde{k}\leq k$ we have
\begin{align*}
\sup_{t\in(0,T]}\left\|t^{\mu_q+1}D_x^\beta\partial_t^{\tilde{k}}\left(t^{\tilde{k}-1}u(t)\right)\right\|_{L^q}
\leq M^{|\beta|+\tilde{k}-\delta}\left(1+\frac{1}{M}\right)\left(|\beta|+\tilde{k}\right)^{|\beta|+\tilde{k}-1}.
\end{align*}
Repeat the above process for $j=1,2,\cdots,k-1$ and we obtain
\begin{align*}
\sup_{t\in(0,T]}\left\|t^{\mu_q+k}D_x^\beta\partial_t^ku(t)\right\|_{L^q}\leq M^{|\beta|+k-\delta}\left(1+\frac{1}{M}\right)^k\left(|\beta|+k\right)^{|\beta|+k-1},
\end{align*}
which implies that inequality \eqref{eq:analytic} holds for $M+1$.
\end{proof}
\begin{remark}[Radius of analyticity]\label{rmk:radius}
Let $\alpha\in\mathbb{N}^3$ be a multi-index with $|\alpha|=m$. According to the equality $\left(\sum_{i=1}^3x_i\right)^m=\sum_{|\alpha|=m}\binom{m}{\alpha}\Pi_{i=1}^3x_i^{\alpha_i}$ we have 
\begin{align}\label{eq:binom}
\alpha !\leq m !\leq 3^m\alpha !.
\end{align}
Notice that for initial datum $u_0$ with $\left\|u_0\right\|_{L^3(\mathbb{R}^3)}$, the constant $M$ is independent of time. Combining \eqref{eq:binom} and Stirling's formula, we have the following estimate for the radius $r(t)$ of analyticity of the mild solution:
\begin{align}\label{eq:radius}
r(t)=\lim_{|\beta|+k\to\infty}\left(\frac{\|D_x^\beta\partial_t^ku(t)\|_{L^\infty}}{\beta ! k!}\right)^{-\frac{1}{|\beta|+k}}\geq C\min\{\sqrt{t},t\},\quad t\in(0,\infty).
\end{align}
Similarly, for the space analytic radius $r_1(t)$, we have $r_1(t)\geq C\sqrt{t}$, and for the time analytic radius $r_2(t)$, we have $r_2(t)\geq Ct$ for any $t>0$.
\end{remark}

\noindent\textbf{Acknowledgements} 
Y. Gao is supported by the National Natural Science Foundation grant 12101521 of China and the Start-Up Fund from the Hong Kong Polytechnic University. X. Xue is supported by the Chinese Natural Science Foundation grants 11731010 and 11671109.

\appendix

\section{Useful lemmas}\label{app}
The next lemma about calculation of multi-indexes comes from \cite[Lemma 2.1]{kahane1969}:
\begin{lemma}\label{lmm:multisequences}
	Let $\kappa\in\mathbb{N}^d$ be a multi-index. If either $\delta$ or $\epsilon<-\frac{1}{2}$, then
	\begin{align}\label{multisequence}
		\sum_{\beta+\gamma=\kappa}\frac{\kappa!}{\beta!\gamma!}\left|\beta\right|^{|\beta|+\delta} \left|\gamma\right|^{|\gamma|+\epsilon}\leq \lambda \left|\kappa\right|^{|\kappa|+max\{\delta, \epsilon\}},
	\end{align}
	where $\lambda$ depends on $\delta$ and $\epsilon$. Here, we used $0^p=1$ for any $p\in\mathbb{R}$.
\end{lemma}

The following useful lemma comes from  \cite[Lemma 3.3]{donghongjie2020jfa}:
\begin{lemma}\label{eq:donglemm} 
	Let $f$ and $g$ be two smooth function on $\mathbb{R}$, for any integer $k\geq 1$, we have
	\begin{multline}\label{eq:timeestimate}
		\partial_t^k\left(t^kf(t)g(t)\right)=\sum_{j=0}^k\binom{k}{j}\partial_t^j\left(t^jf(t)\right)\partial_t^{k-j}\left(t^{k-j}g(t)\right)\\
		-k\sum_{j=0}^{k-1}\binom{k-1}{j}\partial_t^j\left(t^jf(t)\right)\partial_t^{k-1-j}\left(t^{k-1-j}g(t)\right).	
	\end{multline}
\end{lemma}
\begin{proposition}\label{pro:spacetime}
Let  $\beta\in\mathbb{N}^3$ be a multi-index and $L\geq 2$, $k$ be nonnegative integers.  For $0< |\beta|+k\leq L-1$ and $3\leq q<\infty$, if there exist some constants $0<\delta<1$, $M>0$ independent of $|\beta|$ and $k$ such that
\begin{align}\label{fineestimate}
\left\|D_x^\beta\partial_t^k\left(t^ku(t)\right)\right\|_{L^q}\leq M^{|\beta|+k-\delta}\left(|\beta|+k\right)^{|\beta|+k-1}t^{-\frac{|\beta|}{2}-\frac{3}{2}\left(\frac{1}{3}-\frac{1}{q}\right)},~~0< |\beta|+k\leq L-1,
\end{align}
then we have 
\begin{align}\label{fineestimate1}
\left\|D_x^\beta\partial_t^k\left[t^k\mathcal{P}(u\otimes u)(t)\right]\right\|_{L^p}\leq N(1+M^\delta)M^{|\beta|+k-2\delta}\left(|\beta|+k\right)^{|\beta|+k-1}t^{-\frac{|\beta|+1}{2}-\frac{3}{2}\left(\frac{1}{3}-\frac{1}{p}\right)}
\end{align}
for any $3\leq p<\infty$ and $0<|\beta|+k\leq L-1$, and  we also have
\begin{multline}\label{fineestimate2}
\left\|D_x^\beta\partial_t^k\left[t^k\mathcal{P}(u\otimes u)(t)\right]\right\|_{L^p}\leq NM^{|\beta|+k-2\delta}\left(|\beta|+k\right)^{|\beta|+k-1}t^{-\frac{|\beta|+1}{2}-\frac{3}{2}\left(\frac{1}{3}-\frac{1}{p}\right)}\\
+A_1\left\|\left[D_x^\beta\partial_t^k\left(t^ku(t)\right)\right]\otimes u(t)\right\|_{L^p}+A_1\left\|u(t)\otimes \left[D_x^\beta\partial_t^k\left(t^ku(t)\right)\right]\right\|_{L^p}
\end{multline}
for  any $3\leq p<\infty$  and $|\beta|+k=L$. Here, $N>0$ is a constant  independent of $M$, $\beta$ and $k$ and $C$ is a constant independent of $3\leq p<\infty$.
\end{proposition}
\begin{proof}
From the identity \eqref{eq:timeestimate}, the following inequality holds
\begin{equation}\label{eq:timeestimatenorm}
\begin{aligned}
\left\|D_x^\beta\partial_t^k\left[t^k\mathcal{P}(u\otimes u)(t)\right]\right\|_{L^p}\leq A_1\sum_{j=0}^k\binom{k}{j}\left\|D_x^\beta\left[\partial_t^j\left(t^j u(t)\right)\otimes\partial_t^{k-j}\left(t^{k-j}u(t)\right)\right]\right\|_{L^p}\\
+A_1k\sum_{j=0}^{k-1}\binom{k-1}{j}\left\|D_x^\beta\left[\partial_t^j\left(t^j u(t)\right)\otimes\partial_t^{k-1-j}\left(t^{k-1-j} u(t)\right)\right]\right\|_{L^p}\\
\leq A_1\sum_{0\leq\gamma\leq\beta}\binom{\beta}{\gamma}\sum_{j=0}^k\binom{k}{j}\left\| D_x^{\gamma}\partial_t^j\left(t^j u(t)\right)\otimes D_x^{\beta-\gamma}\partial_t^{k-j}\left(t^{k-j}u(t)\right)\right\|_{L^p}\\
+A_1k\sum_{0\leq\gamma\leq\beta}\binom{\beta}{\gamma}\sum_{j=0}^{k-1}\binom{k-1}{j}\left\|D_x^\gamma\partial_t^j\left(t^ju(t)\right)\otimes D_x^{\beta-\gamma}\partial_t^{k-1-j}\left(t^{k-1-j} u(t)\right)\right\|_{L^p}.
\end{aligned}
\end{equation}

\textbf{Step 1:} Assume $0< |\beta|+k\leq L-1$ and we are going to prove \eqref{fineestimate1} in this step.  
From the assumption \eqref{fineestimate}, H\"oder's inequality for $\frac{1}{p}=\frac{1}{p_1}+\frac{1}{p_2}$, we obtain from \eqref{eq:timeestimatenorm} that
\begin{equation*}
\begin{aligned}
&\left\|D_x^\beta\partial_t^k\left[t^k\mathcal{P}(u\otimes u)(t)\right]\right\|_{L^p}\\
\leq&A_1\sum_{\substack{|\gamma|+j\neq 0,\\|\gamma|+j\neq|\beta|+k}}\binom{k}{j}\binom{\beta}{\gamma}\left\|D_x^\gamma\partial_t^j\left(t^ju(t)\right)\right\|_{L^{p_1}}\left\|D_x^{\beta-\gamma}\partial_t^{k-j}\left(t^{k-j}u(t)\right)\right\|_{L^{p_2}}\\
&+A_1k\sum_{\substack{|\gamma|+j\neq 0,\\|\gamma|+j\neq|\beta|+k-1}}\binom{k-1}{j} \binom{\beta}{\gamma}\left\|D_x^\gamma\partial_t^j\left(t^ju(t)\right)\right\|_{L^{p_1}}\left\|D_x^{\beta-\gamma}\partial_t^{k-1-j}\left(t^{k-1-j}u(t)\right)\right\|_{L^{p_2}}\\
&+A_1\|u(t)\|_{L^{p_1}}\left\|D_x^{\beta}\partial_t^{k}\left(t^{k}u(t)\right)\right\|_{L^{p_2}}+A_1\|D_x^\beta\partial_t^k\left(t^ku(t)\right)\|_{L^{p_1}}\left\|u(t)\right\|_{L^{p_2}}\\
&+A_1k\|u(t)\|_{L^{p_1}}\left\|D_x^{\beta}\partial_t^{k-1}\left(t^{k-1}u(t)\right)\right\|_{L^{p_2}}+A_1k\|D_x^\beta\partial_t^{k-1}\left(t^{k-1}u(t)\right)\|_{L^{p_1}}\left\|u(t)\right\|_{L^{p_2}}\\
\end{aligned}
\end{equation*}
\begin{equation*}
\begin{aligned}
\leq&A_1M^{|\beta|+k-2\delta}\sum_{\substack{|\gamma|+j\neq 0,\\|\gamma|+j\neq|\beta|+k}}\binom{k}{j}\binom{\beta}{\gamma}\times\left(|\gamma|+j\right)^{|\gamma|+j-1}\left[|\beta-\gamma|+(k-j)\right]^{|\beta-\gamma|+(k-j)-1}t^{-\frac{|\beta|+1}{2}-\frac{3}{2}\left(\frac{1}{3}-\frac{1}{p}\right)}\\
&+A_1M^{|\beta|+k-1-2\delta}k\sum_{\substack{|\gamma|+j\neq 0,\\|\gamma|+j\neq|\beta|+k-1}}\binom{k-1}{j}\binom{\beta}{\gamma}\\
&\qquad \times\left(|\gamma|+j\right)^{|\gamma|+j-1}\left[|\beta-\gamma|+(k-1-j)\right]^{|\beta-\gamma|+(k-1-j)-1}t^{-\frac{|\beta|+1}{2}-\frac{3}{2}\left(\frac{1}{3}-\frac{1}{p}\right)}\\
&+A_1M^{|\beta|+k-\delta}(|\beta|+k)^{|\beta|+k-1}t^{-\frac{|\beta|+1}{2}-\frac{3}{2}\left(\frac{1}{3}-\frac{1}{p}\right)}+A_1kM^{|\beta|+k-1-\delta}(|\beta|+k-1)^{|\beta|+k-2}t^{-\frac{|\beta|+1}{2}-\frac{3}{2}\left(\frac{1}{3}-\frac{1}{p}\right)}
\end{aligned}
\end{equation*}
for some constant $A_1$. Consider $\bar\beta=(\beta_1,\beta_2,\beta_3,k)\in\mathbb{N}^{4}~\textrm{and}~ \tilde\beta=(\beta_1,\beta_2,\beta_3,k-1)\in\mathbb{N}^{4}$. Hence, $|\bar\beta|=|\beta|+k$ and $|\tilde\beta|=|\beta|+k-1$, and combining Lemma \ref{lmm:multisequences} gives
\begin{small}
\begin{align*}
&\sum_{\substack{|\gamma|+j\neq 0,\\|\gamma|+j\neq|\beta|+k}}\binom{k}{j}\binom{\beta}{\gamma}\left(|\gamma|+j\right)^{|\gamma|+j-1}\left[|\beta-\gamma|+(k-j)\right]^{|\beta-\gamma|+(k-j)-1}\\
\leq &\sum_{j=0}^k\sum_{0\leq\gamma\leq\beta}\frac{\beta_1!\cdots\beta_3!k!}{(\beta_1-\gamma_1)!\gamma_1!\cdots(\beta_3-\gamma_3)!\gamma_3!(k-j)!j!}\left(|\gamma|+j\right)^{|\gamma|+j-1}\left[|\beta-\gamma|+(k-j)\right]^{|\beta-\gamma|+(k-j)-1}\\
=&\sum_{0\leq\bar\gamma\leq\bar\beta}\binom{\bar\beta}{\bar\gamma}\left|\bar\gamma\right|^{|\bar\gamma|-1}\left|\bar\beta-\bar\gamma\right|^{|\bar\beta-\bar\gamma|-1}\leq\lambda|\bar\beta|^{|\bar\beta|-1},
\end{align*}
\end{small}
and
\begin{multline*}
\sum_{\substack{|\gamma|+j\neq 0,\\|\gamma|+j\neq|\beta|+k-1}}\binom{k-1}{j}\binom{\beta}{\gamma}\left(|\gamma|+j\right)^{|\gamma|+k-1}\left[|\beta-\gamma|+(k-1-j)\right]^{|\beta-\gamma|+(k-1-j)-1}\\
\leq \sum_{0\leq\tilde\gamma\leq\tilde\beta}\binom{\tilde\beta}{\tilde\gamma}\left|\tilde\gamma\right|^{|\tilde\gamma|-1}\left|\tilde\beta-\tilde\gamma\right|^{|\tilde\beta-\tilde\gamma|-1}\leq\lambda|\tilde\beta|^{|\tilde\beta|-1}.
\end{multline*}
Combining all the above inequalities, we obtain
\begin{align*}
&\left\|D_x^\beta\partial_t^k\left[t^k\mathcal{P}(u\otimes u)(t)\right]\right\|_{L^p}\leq A_1\lambda M^{|\beta|+k-2\delta}(|\beta|+k)^{|\beta|+k-1}t^{-\frac{|\beta|+1}{2}-\frac{3}{2}\left(\frac{1}{3}-\frac{1}{p}\right)}\\
&+A_1\lambda k M^{|\beta|+k-1-2\delta}k(|\beta|+k-1)^{|\beta|+k-2}t^{-\frac{|\beta|+1}{2}-\frac{3}{2}\left(\frac{1}{3}-\frac{1}{p}\right)}\\
&+A_1M^{|\beta|+k-\delta}(|\beta|+k)^{|\beta|+k-1}t^{-\frac{|\beta|+1}{2}-\frac{3}{2}\left(\frac{1}{3}-\frac{1}{p}\right)}+A_1kM^{|\beta|+k-1-\delta}(|\beta|+k-1)^{|\beta|+k-2}t^{-\frac{|\beta|+1}{2}-\frac{3}{2}\left(\frac{1}{3}-\frac{1}{p}\right)}\\
&\leq(2A_1\lambda+A_1+A_1M^\delta)M^{|\beta|+k-2\delta}\left(|\beta|+k\right)^{|\beta|+k-1}t^{-\frac{|\beta|+1}{2}-\frac{3}{2}\left(\frac{1}{3}-\frac{1}{p}\right)}.
\end{align*}
Let $N=2\lambda A_1+A_1$ and we obtain \eqref{fineestimate1}.

\textbf{Step 2:} Consider the case: $|\beta|+k=L$. Notice that the terms for $|\gamma|=j=0$ and $\gamma=\beta,~~j=k$ in the summation $I_1$ of \eqref{eq:timeestimatenorm} are 
\[
\left\|D_x^\beta\partial_t^k\left(t^ku(t)\right)\otimes u(t)\right\|_{L^p}+\left\|u(t)\otimes D_x^\beta\partial_t^k\left(t^ku(t)\right)\right\|_{L^p}.
\]
Take these two terms out and by the same arguments as in Step 1 for the rest terms in \eqref{eq:timeestimatenorm}, we obtain \eqref{fineestimate2}.
\end{proof}

\bibliographystyle{plain}
\bibliography{bibofNS}

\end{document}